\documentclass[11pt]{amsart}
\usepackage{amsmath,amsfonts,amsthm,mathrsfs,amssymb,amscd,comment,enumitem,amsxtra,url,tikz-cd}
\input{xypic}
\xyoption{all}

\usepackage{xcolor} 
\colorlet{mdtRed}{red!50!black}
\definecolor{dblue}{rgb}{0,0,.6}
\usepackage[colorlinks]{hyperref}
\hypersetup{linkcolor=blue,citecolor=dblue,filecolor=dullmagenta,urlcolor=mdtRed}

\numberwithin{equation}{section}
\newtheorem{theorem}[equation]{Theorem}
\newtheorem{corollary}[equation]{Corollary}
\newtheorem{lemma}[equation]{Lemma}
\newtheorem{proposition}[equation]{Proposition}

\newtheorem*{theorem*}{Theorem}
\newtheorem*{corollary*}{Corollary}
\newtheorem*{proposition*}{Proposition}

\theoremstyle{definition}
\newtheorem{remark}[equation]{Remark}

\def\subsection{
	\refstepcounter{equation}
	\noindent {\bf \arabic{section}.\arabic{equation}.}
}

\newcommand{\Z}{\mathbb{Z}}
\newcommand{\C}{\mathbb{C}}

\newcommand{\R}{\mathbb{R}}

\newcommand{\ms}[1]{\mathscr{#1}}
\newcommand{\mb}[1]{\mathbb{#1}}
\newcommand{\mc}[1]{\mathcal{#1}}
\renewcommand{\t}[1]{\tilde{#1}}

\usepackage{marginnote}
\usetikzlibrary{calc}

\tikzset{
	symbol/.style={
		draw=none,
		every to/.append style={
			edge node={node [sloped, allow upside down, auto=false]{$#1$}}}
	}
}

\begin{document}
	
	\title[Nef and Effective cones of Some Quot Schemes]
	{Nef and Effective cones of Some Quot Schemes}

	\author[C. Gangopadhyay]{Chandranandan Gangopadhyay} 
	
	\address{Department of Mathematics, Shiv Nadar University, NH91, Tehsil Dadri, Greater Noida, Uttar
Pradesh 201314, India} 
	
	\email{chandranandan.g@snu.edu.in} 
	
	\author[R. Sebastian]{Ronnie Sebastian} 
	
	\address{Department of Mathematics, Indian Institute of Technology Bombay, Powai, Mumbai 400076, Maharashtra, India.} 
	
	\email{ronnie@iitb.ac.in} 
	
	

	\begin{abstract}
	 Let $C$ be a smooth projective curve over $\C$
	 of genus $g(C)\geqslant 3$ (respectively, $g(C)=2$). Fix integers $r,k$ 
	 such that $2\leqslant k\leqslant r-2$, 
	 (respectively, $3\leqslant k\leqslant r-2$). 
	 Let $\mc Q:={\rm Quot}_{C/ \mb C}(\mc O^{\oplus r}_C, k, d)$ 
	 be the Quot scheme parametrizing rank $k$ and 
	 degree $d$ quotients of the trivial bundle of 
	 rank $r$. Let $\mc Q_L$ denote the closed subscheme 
	 of the Quot scheme parametrizing quotients such
	 that the quotient sheaf has determinant $L$. 
	 It is known that $\mc Q_L$ is an integral, 
	 normal, local complete
	 intersection, locally factorial scheme of Picard rank 2,
	 when $d\gg0$.
	 In this article we compute the nef cone, effective cone and 
	 canonical divisor of this variety when $d\gg0$. 
	 We show this variety is Fano iff $r=2k+1$. 
	\end{abstract}

	\maketitle

	\section{Introduction}
	
	Let $C$ be a smooth projective curve over the 
	field of complex numbers $\C$. Fix integers 
	$0<k<r$. Let $E$ be a locally free sheaf of rank $r$ on $C$. 
	Let ${\rm Quot}_{C/ \mb C}(E, k, d)$ 
	denote Grothendieck's Quot scheme parametrizing rank $k$ and 
	degree $d$ quotients of the bundle $E$.
	Quot schemes are
	ubiquitous in algebraic geometry. They are useful in many moduli questions, for example
	GIT constructions of moduli spaces of vector bundles and sheaves. They have interesting
	geometric, representation-theoretic and enumerative properties. 
	They admit natural generalisations; for example, moduli spaces of stable maps to homogeneous space bundles.
	It is therefore natural to investigate some basic and 
	important questions about Quot schemes
	of bundles over curves. 
	
	Let ${\rm Quot}_{C/ \mb C}(\mc O^{\oplus r}_C, k, d)$ 
	denote the Quot scheme parametrizing rank $k$ and 
	degree $d$ quotients of the trivial bundle of 
	rank $r$. When $C\cong \mb P^1$, Stromme, in \cite{Str}, 
	proved that ${\rm Quot}_{\mb P^1/ \mb C}(\mc O^{\oplus r}_C, k, d)$
	is a smooth projective variety of Picard rank 
	$2$ and computed its nef cone.  In \cite{Jow}, the author computed 
	the effective cone of ${\rm Quot}_{\mb P^1/ \mb C}(\mc O_{\mb P^1}^{\oplus n},k,d)$. 
	In \cite{Ve}, the author determined the movable
	cone of ${\rm Quot}_{\mb P^1/ \mb C}(\mc O_{\mb P^1}^{\oplus n},k,d)$ and 
	the stable base locus decomposition of the effective cone.
	In \cite{Ito}, the author studied the birational geometry 
	of ${\rm Quot}_{\mb P^1/ \mb C}(\mc O_{\mb P^1}^{\oplus n},k,d)$.
	
	For curves of higher genus, the space ${\rm Quot}_{C/ \mb C}(\mc O^{\oplus r}_C, k, d)$ 
	was studied in \cite{BDW} and it was proved 
	that when $d\gg0$, this Quot scheme is irreducible 
	and generically smooth. This was generalized by Popa 
	and Roth \cite{PR}, who proved the same result for 
	${\rm Quot}_{C/ \mb C}(E, k, d)$, the Quot 
	scheme parametrizing quotients of a 
	vector bundle $E$. See also \cite{Goller},
	\cite{CCH-Isotropic}, \cite{CCH-Lagrangian}, \cite{RS24} 
	for similar results on other variations of this Quot scheme. 
	In \cite{Ras24}, the author generalizes the above mentioned
	result of Popa and Roth to the case of nodal curves. 
	
	Following the work of \cite{PR}, it is natural to ask 
	if the results of \cite{Str}, \cite{Jow}, \cite{Ve} and \cite{Ito} 
	can be generalized to the higher genus case. A more basic question  
	is to know if these Quot schemes are integral, and if one can compute
	their divisor class group and their Picard group. 
	Motivated by these questions, 
	in \cite{gs22}, the authors prove that for $d\gg 0$, 
	${\rm Quot}_{C/ \mb C}(E, k, d)$ is integral, local complete intersection, 
	normal and locally factorial scheme
	and compute its Picard group. 
	Consider the determinant map
	$${\rm det}:{\rm Quot}_{C/ \mb C}(E, k, d) \to {\rm Pic}^d(C)\,,$$
	which sends a closed point $[E\to F]$ to the 
	determinant line bundle $[{\rm det}(F)]$. For 
	$[L]\in {\rm Pic}^d(C)$ let us denote the fiber over 
	$[L]$ by ${\rm Quot}_{C/ \mb C}(E, k, d)_L$. By definition, 
	${\rm Quot}_{C/ \mb C}(E, k, d)_L$ parametrizes 
	quotients $[E\to F]$ such that rank of $F$ is $k$
	and determinant of $F$ is $L$. 
	With some mild assumptions on the 
	genus of $C$ and $k$, it was proved in \cite{gs22} 
	that when $d\gg 0$ the 
	scheme ${\rm Quot}_{C/ \mb C}(E, k, d)_L$ is integral, local complete intersection, 
	normal and locally factorial scheme and 
	moreover, its Picard group is isomorphic to $\Z\oplus \Z$. 
	See the next section 
	for precise statements of the main results in \cite{gs22}. 
	
	
	For ease of notation, let $\mc Q_L$ denote ${\rm Quot}_{C/ \mb C}(\mc O^{\oplus r}_C, k, d)_L$. 
	Let ${\rm N}^1(\mc Q_L)$
	denote the Neron-Severi group of $\mc Q_L$.
	Let ${\rm Nef}(\mc Q_L)\subset {\rm N}^1(\mc Q_L)\otimes_\Z\R$ denote the 
	cone of nef divisors. 
	Similarly, we have the cone of effective divisors, which we denote by 
	${\rm Eff}(\mc Q_L)$ and the 
	movable cone of divisors which we denote 
	${\rm Mov}(\mc Q_L)$.
	In this article, we 
	compute these cones, thereby 
	generalizing the results in \cite{Str}, \cite{Jow}, \cite{Ve}. 
	Note that in the case of $\mb P^1$, this 
	scheme is just ${\rm Quot}_{\mb P^1/ \mb C}(\mc O^{\oplus r}_C, k, d)$. 
	It is surprising to note the similarity between $\mc Q_L$ and 
	${\rm Quot}_{\mb P^1/ \mb C}(\mc O^{\oplus r}_C, k, d)$.
	In view of this similarity, and the results of \cite{Ito}, 
	it is natural to ask if 
	$\mc Q_L$ is a Mori dream space. 
	
	We now state the main results of this article. 
	Let $\alpha$ and $\beta_{d+g-1}$ denote the boundaries of the 
	nef cone.  
	For the definitions 
	of these line bundles,
	see the discussion before \eqref{def alpha}, 
	the discussion before Lemma \ref{def beta_M} and Remark \ref{beta_m}.
	We also
	compute the class of the canonical divisor of 
	$\mc Q_L$ in terms of $\alpha$ and $\beta_{d+g-1}$. 
	Combining this with Theorem \ref{main theorem nef} we give a necessary 
	and sufficient condition for $\mc Q_L$ to be Fano.
	This question, regarding when $\mc Q_L$ is Fano, 
	was raised by Pieter Belmans in his Blog 
	(see  \href{https://pbelmans.ncag.info/blog/2022/11/09/fortnightly-links-160/}{Fortnightly links (160)}\footnote{\href{https://pbelmans.ncag.info/blog/2022/11/09/fortnightly-links-160/}{https://pbelmans.ncag.info/blog/2022/11/09/fortnightly-links-160/}}).

	In the following Theorem, 
	for the definitions of the curves $D_1$ and $D_2$, 
	see the proofs of Proposition \ref{alpha nef} and Proposition \ref{beta nef}.

	\begin{theorem}\label{thm intro}
		Assume one of the following two holds:
		\begin{itemize}
			\item $g(C)\geqslant 3$ and $2\leqslant k\leqslant r-2$, or
			\item $g(C)=2$ and $3\leqslant k\leqslant r-2$.
		\end{itemize}
		Let $d\gg0$. Then we have the following results:
		\begin{enumerate}[label=(\Alph*)]
			\item \label{intro 1} ${\rm Pic}(\mc Q_L)$ 
			is generated by the line bundles
			$\alpha$, $\beta_{d+g-1}$. Both these are globally generated, nef but not ample.
			In particular, 
			$${\rm Nef}(\mc Q_L)=\R_{\geqslant0}\alpha + \R_{\geqslant0}\beta_{d+g-1}\,.$$
			The boundaries of the cone of effective curves are given by the classes 
			of the curves $D_1$ and $D_2$.
			\item  \label{intro 2} 
			The effective cone of $\mc Q_L$ is given by 
			$${\rm Eff}(\mc Q_L)=\R_{\geqslant0}(d(k+1)\alpha-k\beta_{d+g-1}) +
			\R_{\geqslant0}(-d(r-k-1)\alpha+(r-k)\beta_{d+g-1})\,.$$
			The cone ${\rm Mov}(\mc Q_L)={\rm Eff}(\mc Q_L)$.
			\item  \label{intro 3} 
			Let $\omega_{\mc Q_L}$ denote the canonical 
			divisor of $\mc Q_L$. Then 
			$$\omega_{\mc Q_L}=[d(r-2k-2)+r(g-1)]\alpha+(2k-r)\beta_{d+g-1}\,.$$
			Thus, $\mc Q_L$ is Fano iff $r=2k+1$. 
		\end{enumerate}
	\end{theorem}
	Statement \ref{intro 1} is proved in Theorem \ref{main theorem nef}, statement \ref{intro 2}
	is proved in Theorem \ref{main theorem effective cone} and  statement \ref{intro 3} is proved in
	Theorem \ref{main theorem canonical}.

	\section{Recollection of some results from \cite{gs22}}
	
	Let $C$ be a smooth projective curve over the 
	field of complex numbers $\C$. We shall denote the 
	genus of $C$ by $g(C)$. Throughout this article we shall
	assume that $g(C)\geqslant 2$. 
	Let $E$ be a 
	locally free sheaf on $C$ of rank $r$ and degree $e$.
	In the latter sections we will be considering only the case 
	$E=\mc O_C^{\oplus r}$, whence, $e=0$. 
	Let $k$ be an integer such that $0<k<r$.
	Throughout this article  
	\begin{equation}\label{def Q}
		{\rm Quot}_{C/\C}(E,k,d)
	\end{equation}
	will denote the Quot scheme of quotients of $E$ of rank $k$
	and degree $d$. There is a map 
	\begin{equation}\label{def det}
		{\rm det}\colon{\rm Quot}_{C/\C}(E,k,d)\to {\rm Pic}^d(C)\,,
	\end{equation}
	see \cite[equation (2.5)]{gs22}. This map sends a closed point $[E\to F]\in {\rm Quot}_{C/\C}(E,k,d)$
	to $[{\rm det}(F)]\in {\rm Pic}^d(C)$. 
	Let $L$ be a line bundle on $C$ of degree $d$ and let 
	$${\rm Quot}_{C/\C}(E,k,d)_L:={\rm det}^{-1}([L])$$
	be the scheme theoretic fiber over the point $[L]\in {\rm Pic}^d(C)$.
	We recall the main results in \cite{gs22}.

	\begin{enumerate}[label=(\Alph*)]
		\item \label{point 2}
		\cite[Theorem 3.3, Corollary 3.5]{gs22}.
		First consider the case $k=r-1$. 
		Assume $d>2g-2+e-\mu_{\rm min}(E)$.
		There is a locally free sheaf $\mc E$ on ${\rm Pic}^{e-d}(C)$ such that
		the following holds. 
		We have an isomorphism of schemes over ${\rm Pic}^{e-d}(C)$,
		$\mb P(\mc E^{\vee})\xrightarrow{\sim} {\rm Quot}_{C/\C}(E,r-1,d)$.
		In particular, under the above assumption on $d$, the space 
		${\rm Quot}_{C/\C}(E,r-1,d)$ is smooth and 
		${\rm Pic}({\rm Quot}_{C/\C}(E,r-1,d))\cong {\rm Pic}({\rm Pic}^0(C))\times \mb Z$.
	\end{enumerate}
	
	\noindent 
	Next we consider the case when $k\leqslant r-2$. There is a number $d_0(E,k)$, 
	which depends only on $E$ and $k$, such that the following statements hold. 
	Let $d\geqslant d_0(E,k)$.  
	\begin{enumerate}
		\item \cite[Theorem 6.3]{gs22} 
		Then ${\rm det}\colon{\rm Quot}_{C/\C}(E,k,d)\longrightarrow {\rm Pic}^d(C)$ is a flat map.
		Further, ${\rm Quot}_{C/\C}(E,k,d)$ is local complete intersection scheme which is an 
		integral and normal variety and is locally factorial.
		
		\item \cite[Theorem 9.1]{gs22} Let $k=1$ and $r\geqslant 3$ (the case $k=1$ and $r=2$
		is dealt with in the case $k=r-1$ above). Then 
		${\rm Pic}({\rm Quot}_{C/\C}(E,k,d))\cong {\rm Pic}({\rm Pic}^d(C))\times \Z\times \Z$.
		\item \cite[Theorem 8.7]{gs22} Let $k\geqslant 2,g(C)\geqslant 2$.
		Then ${\rm Quot}_{C/\C}(E,k,d)_L$ 
		is a local complete intersection scheme, which
		is also integral, normal and locally factorial.
		\item \label{Picard rank 2}\cite[Theorem 8.9]{gs22} Assume one of the following two holds
		\begin{itemize}
			\item $k\geqslant 2$ and $g(C)\geqslant 3$, or
			\item $k\geqslant 3$ and $g(C)=2$.
		\end{itemize}
		We have isomorphisms 
		$${\rm Pic}({\rm Quot}_{C/\C}(E,k,d)_L)\cong {\rm Pic}(M^s_{k,L})\times \Z \cong \Z \times \Z\,.$$
		\item \cite[Theorem 9.1]{gs22} Let $k=1$ and $r\geqslant 3$. Then 
		${\rm Pic}({\rm Quot}_{C/\C}(E,k,d)_L)\cong \Z\times \Z$. 
	\end{enumerate}
	
	In view of point \eqref{Picard rank 2}, it becomes a particularly interesting 
	question to investigate the nef cone of the scheme ${\rm Quot}_{C/\C}(E,k,d)_L$. The purpose 
	of this article is to investigate this question when $E$ is the trivial bundle of 
	rank $r$. Before we proceed we mention a few points from \cite{gs22} which we shall use. 
	The discussion in this paragraph assumes that $d\gg0$. 
	The ``good locus" of the Quot scheme is defined to 
	be the set of points (see \cite[Definition 4.4]{gs22})
	$${\rm Quot}_{C/\C}(E,k,d)_g:=\{[E\to F]\,\vert \, H^1(E^\vee\otimes F)=0 \}\,.$$
	Let $A$ be a locally closed subset of the 
	Quot scheme. Then the good locus of 
	$A$, denoted $A_g$ is defined to be the 
	subset $A\cap {\rm Quot}_{C/\C}(E,k,d)_g$.
	An important property of the good locus is 
	that the morphism ${\rm det}$ in equation \eqref{def det} restricted 
	to the good locus is a smooth morphism. In particular, taking $A={\rm Quot}_{C/\C}(E,k,d)_L$,
	we get the locus ${\rm Quot}_{C/\C}(E,k,d)_{g,L}$. Another subset of the Quot scheme 
	which will be used is the locus of stable quotients, that is, 
	$${\rm Quot}_{C/\C}(E,k,d)^s:=\{[E\to F]\,\vert \, F\text{ is stable} \}\,.$$
	Similarly, we define ${\rm Quot}_{C/\C}(E,k,d)^s_L$. 
	We have inclusions 
	$${\rm Quot}_{C/\C}(E,k,d)^s_L\subset {\rm Quot}_{C/\C}(E,k,d)_{g,L}\subset {\rm Quot}_{C/\C}(E,k,d)_L\,.$$
	If $Y\subset X$ is locally closed, then we denote ${\rm codim}(Y,X)={\rm dim}(X)-{\rm dim}(Y)$.
	In the proof of \cite[Theorem 8.9]{gs22}, it is proved that 
	$${\rm codim}({\rm Quot}_{C/\C}(E,k,d)_L\setminus {\rm Quot}_{C/\C}(E,k,d)^s_L,
	{\rm Quot}_{C/\C}(E,k,d)_L)\geqslant 2\,.$$

	\section{Notation}
	Fix a point $[L]\in {\rm Pic}^d(C)$.
	For the remainder of this article we impose the following conventions:
	\begin{itemize}
		\item Denote by $\mc Q:={\rm Quot}_{C/\C}(\mc O_C^{\oplus r},k,d)$ and denote by 
		$\mc Q_L:={\rm Quot}_{C/\C}(\mc O_C^{\oplus r},k,d)_L.$
		\item Assume one of the following two holds:
		\begin{itemize}
			\item[$\dagger$] $g(C)\geqslant 3$ and $2\leqslant k\leqslant r-2$, or
			\item[$\dagger$] $g(C)=2$ and $3\leqslant k\leqslant r-2$.
		\end{itemize}
		\item If $Y\subset X$ is locally closed, then we denote ${\rm codim}(Y,X)={\rm dim}(X)-{\rm dim}(Y)$.
		\item For a scheme $T$, we shall denote by 
		$p_C\colon C\times T\to C$ the projection. The projection onto the second
		factor will be denoted by $p_2$. 
		\item Let
		\begin{equation}\label{universal seq}
			0\to \mc K\to p_C^*\mc O^{\oplus r}_{C} \to \mc F\to 0
		\end{equation}
		denote the universal quotient on $C\times \mc Q$. We will abuse 
		notation and use the same notation to denote its restriction to 
		$C\times \mc Q_L$. 
		\item We shall assume that $d\gg0$. In particular, $d\geqslant d_0(E,k)$.
	\end{itemize}
	\begin{remark}\label{pull-push}
		At several places we will use the following easy consequence of cohomology
		and base change. Let $f\colon X\to Y$ be a projective morphism
		and let $\mc F$ be a coherent sheaf on $X$ which is flat over $Y$. Suppose 
		$h^1(X_y,\mc F_y)=0$ for all closed points $y\in Y$. Then $f_*(\mc F)$ is locally free. 
		Let $g\colon Y'\to Y$ be a morphism
		and consider the Cartesian square
		\[\xymatrix{
			X'\ar[r]^{g'}\ar[d]_{f'} & X\ar[d]\\
			Y'\ar[r]^g & Y
		}
		\] 
		Then the natural map $g^*f_*(\mc F)\to f'_*g'^*(\mc F)$
		is an isomorphism. 
	\end{remark}
	
	\section{Nef cone of $\mc Q_L$}
	
	Recall the universal sequence \eqref{universal seq} on $C\times \mc Q_L$. Applying $\wedge^{r-k}$ 
	to the inclusion $\mc K\subset  p_C^*\mc O^{\oplus r}_{C}$, we get an 
	inclusion, which sits in a short exact sequence  
	\begin{equation}\label{wedge r-k universal}
		0\to \wedge^{r-k}\mc K\to \wedge^{r-k}\Big(p_C^*\mc O^{\oplus r}_{C}\Big)\to \mc F'\to 0\,.
	\end{equation}
	Let $q\in \mc Q_L$ be a closed point. The restriction of the 
	map in \eqref{wedge r-k universal} to $C\times q$ is the same 
	as restricting the map \eqref{universal seq} to $C\times q$ 
	and then applying $\wedge^{r-k}$. From this it easily follows 
	that the restriction of \eqref{wedge r-k universal} to $C\times q$ 
	is an inclusion. For each point $q\in \mc Q_L$,
	the sheaf $\wedge^{r-k}\mc K_q:=\wedge^{r-k}\mc K\vert_{C\times q}\cong L^{-1}$.
	Thus, it follows that the rank and degree of $\mc F'_q$ 
	are constant. As $\mc Q_L$ is an integral scheme, 
	it follows that $\mc F'$ of \eqref{wedge r-k universal}
	is flat over $\mc Q_L$. The quotient 
	$$\wedge^{r-k}\Big(p_C^*\mc O^{\oplus r}_{C}\Big)\to \mc F'$$
	on $C\times \mc Q_L$ gives rise to a morphism
	from $\mc Q_L$ to the quot scheme  
	$${\rm Quot}_{C/\C}\Big(\wedge^{r-k}\Big(\mc O^{\oplus r}_{C}\Big), 
	\binom{r}{r-k}-1,d\Big)\,.$$ 
	Moreover, for each $q$, the cokernel $\mc F'_q$
	has determinant $L$. It follows that the image of the composite map 
	$$\mc Q_L \to {\rm Quot}_{C/\C}\Big(\wedge^{r-k}\Big(\mc O^{\oplus r}_{C}\Big), 
	\binom{r}{r-k}-1,d\Big) \stackrel{\rm det}{\to } {\rm Pic}^d(C)$$
	is, at least set theoretically, the closed point $[L]$. As $\mc Q_L$
	is an integral scheme, it follows that the image lands in the scheme 
	theoretic fiber ${\rm Quot}_{C/\C}\Big(\wedge^{r-k}\Big(\mc O^{\oplus r}_{C}\Big), 
	\binom{r}{r-k}-1,d\Big)_L$. By \ref{point 2} 
	(that is, \cite[Theorem 3.3, Corollary 3.5]{gs22}) it follows 
	that we get a map from 
	\begin{equation}\label{def f}
		f\colon \mc Q_L\to \mb P(\mc E^\vee_{L^{-1}})
	\end{equation}
	(note that $e=0$). 
	For ease of notation we denote $\mb P(\mc E^\vee_{L^{-1}})$ by $\mb P$.
	Define 
	\begin{equation}\label{def alpha}
		\alpha:=f^*\mc O_{\mb P}(1)\,.
	\end{equation}
	
	\begin{proposition}\label{alpha nef}
		The line bundle $\alpha$ is nef but not ample. 
	\end{proposition}
	\begin{proof}
		It is clear that $\alpha$ is nef. To show it is not ample,
		it suffices to find a curve $D_1\subset \mc Q_L$ such that 
		the restriction of $\alpha$ to $D_1$ is trivial. 
		
		We begin with describing the map $f$ in \eqref{def f} in some more detail. 
		Let $p_2\colon C\times \mc Q_L\to \mc Q_L$ denote the projection. 
		By the Seesaw Theorem, there is a line bundle $M$ on $\mc Q_L$
		such that $\wedge^{r-k}\mc K\cong p_C^*L^{-1}\otimes p_2^*M$.
		Tensoring \eqref{wedge r-k universal} with $p_C^*L$ 
		we get the following exact sequence of sheaves on $C\times \mc Q_L$
		$$0\to p_2^*M \to \Big[\wedge^{r-k}\Big(p_C^*\mc O^{\oplus r}_{C}\Big)\Big]\otimes p_C^*L\to 
		\mc F'\otimes p_C^*L\to 0\,.$$
		Applying $p_{2*}$ we get the following exact sequence of sheaves 
		on $\mc Q_L$
		$$0\to M\to H^0(C,\Big[\wedge^{r-k}\Big(\mc O^{\oplus r}_{C}\Big)\Big]\otimes L)\otimes \mc O_{\mc Q_L}
		\to p_{2*}(\mc F'\otimes p_C^*L)\to H^1(C,\mc O_C)\otimes M\to 0\,.$$
		The last term on the right is 0 as the degree of $L$ is 
		$d$, which is assumed to be very large. It follows that $p_{2*}(\mc F'\otimes p_C^*L)$ is locally free.
		Taking dual of the above sequence, we get a surjection 
		$$H^0(C,\wedge^{r-k}\Big[\wedge^{r-k}\Big(\mc O^{\oplus r}_{C}\Big)\Big]\otimes L)^\vee\otimes \mc O_{\mc Q_L}\to M^{-1}\,,$$
		which defines the map $f$ to 
		$$\mb P(H^0(C,\wedge^{r-k}\Big[\wedge^{r-k}\Big(\mc O^{\oplus r}_{C}\Big)\Big]\otimes L)^\vee)\,.$$
		It is clear that the pullback of $\mc O_{\mb P}(1)$ is $M^{-1}$. 
		Thus, $\alpha=M^{-1}$.
		
		To construct the curve $D_1$, fix a closed point $x\in C$ and 
		a subsheaf $K'\subset \mc O_C^{\oplus r}$
		with ${\rm det}(K')=L^{-1}\otimes \mc O_C(x)$. 
		Let $D_1\subset \mb P(K'_x)$ be a line in the projective
		space associated to the vector space $K'_x$. Let $\iota_x\colon D_1\to C\times D_1$
		denote the map $t\mapsto (x,t)$. 
		On $C\times D_1$ 
		we have the surjection 
		$$p_C^*K'\to \iota_{x*}\iota_x^*p_C^*K'=\iota_{x*}(K'_x\otimes \mc O_{D_1})\to \iota_{x*}(\mc O_{D_1}(1))\,.$$
		Let $\t K_1$ be the kernel of this surjection. As $\iota_{x*}(\mc O_{D_1}(1))$
		is flat over $D_1$, it follows that $\t K_1$ is flat over $D_1$. 
		We have the following commutative diagram over $C\times D_1$ 
		in which all three term sequences are exact 
		\begin{equation}\label{family D_1}
			\xymatrix{
				0\ar[r]& \t K_1\ar[r]\ar[d] & p_C^*\mc O^{\oplus r}\ar[r]\ar@{=}[d] & \mc G_1\ar[r]\ar[d] & 0\\
				0\ar[r]& p_C^*K'\ar[r]\ar[d] & p_C^*\mc O^{\oplus r}\ar[r] & p_C^*F'\ar[r] & 0\\
				&\iota_{x*}(\mc O_{D_1}(1))
			}
		\end{equation}
		It follows easily that for $t\in D_1$, the rank and degree of 
		$\mc G_{1,t}$ are independent of $t$. Thus, $\mc G_1$ is flat over $D_1$. 
		Note that 
		\begin{align*}
			\wedge^{r-k}\t K_1&\cong {\rm det}(p_C^*K')\otimes {\rm det}(\iota_{x*}(\mc O_{D_1}(1)))^{-1}\\
			&\cong  p_C^*(L^{-1}\otimes \mc O_C(x))\otimes p_C^*\mc O_C(-x)\\
			&\cong p_C^*L^{-1}\,.
		\end{align*}
		It easily follows that we get a morphism $D_1\to \mc Q_L$, which 
		is an inclusion on closed points. 
		
		As $p_C^*L^{-1}=p_C^*L^{-1}\otimes p_2^*\mc O_{D_1}$, 
		from the description of the morphism $f$, it is clear that the restriction of 
		$f^*(\mc O_{\mb P}(1))$ to $D_1$ is the trivial bundle. This shows that $\alpha$
		is not ample. 
	\end{proof}
	
	Let $\mc Q_L^s$ denote the locus of quotients $[\mc O_C^{\oplus r}\to F]$
	such that $F$ is a stable bundle. 
	To construct our second nef class, we shall first construct a 
	vector bundle on $\mc Q_L^s$. The determinant of this gives a line 
	bundle on $\mc Q_L^s$. Under our hypothesis, it follows that 
	${\rm codim}(\mc Q_L\setminus \mc Q^s_L,\mc Q_L) \geqslant 2$ (see the discussion before equation (8.10)
	in \cite{gs22}). Thus, the line bundle constructed 
	on $\mc Q^s_L$ extends uniquely to a line bundle on $\mc Q_L$. We will
	show that this line bundle is nef.

	Let $M$ be a line bundle on $C$ of degree $m$ such that 
	\begin{equation}\label{condition deg m}
		\frac{d}{k}+m>2g-2\,. 
	\end{equation}
	Recall the universal sheaf $\mc F$ from \eqref{universal seq}.
	Let $p_{2}$ denote the projection $C\times \mc Q^s_L\to \mc Q^s_L$. 
	Consider the sheaf $p_{2*}(\mc F\otimes p_{C}^*M)$ on
	$\mc Q^s_L$. For a point $q\in \mc Q^s_L$,
	we have $h^1(C,\mc F_q\otimes M)=0$ iff ${\rm hom}(\mc F_q\otimes M,\omega_C)=0$.
	Assume $\mu_{\rm min}(\mc F_q\otimes M)>\mu_{\rm max}(\omega_C)$, that is,
	if \eqref{condition deg m} holds. By \cite[Lemma 1.3.3]{HL}, 
	it follows that ${\rm hom}(\mc F_q\otimes M,\omega_C)=0$.
	Thus, if $m$ is such that this inequality holds, then it follows 
	easily, using cohomology and base change \cite[Theorem 12.11]{Ha},
	that the sheaf $p_{2*}(\mc F\otimes p_{C}^*M)$ on
	$\mc Q^s_L$ is locally free. The determinant of this locally free sheaf gives 
	a line bundle on $\mc Q^s_L$, which extends uniquely to a line bundle on 
	$\mc Q_L$. We denote this line bundle by $\beta_M$. 
	
	\begin{lemma}\label{def beta_M}
		Let $M$ and $M'$ have the same degree $m$ such that \eqref{condition deg m} holds.
		Then $\beta_M$ is numerically equivalent to $\beta_{M'}$.
	\end{lemma}
	\begin{proof}
		Let $\mc P$ be a Poincare bundle on $C\times {\rm Pic}^m(C)$.
		Let $p_{ij}$ denote the projection maps from $C\times \mc Q^s_L\times {\rm Pic}^m(C)$. 
		Consider the sheaf $p_{23*}(\mc F\otimes p_{13}^*\mc P)$ on
		$\mc Q^s_L\times {\rm Pic}^m(C)$. For a point $(q,M)\in \mc Q^s_L\times {\rm Pic}^m(C)$,
		we have $h^1(C,\mc F_q\otimes M)=0$.
		It follows 
		easily that the sheaf $p_{23*}(\mc F\otimes p_{13}^*\mc P)$ on
		$\mc Q^s_L\times {\rm Pic}^m(C)$ is locally free. Let us denote the 
		determinant of this sheaf by $\mc R$. 
		It can be easily seen, for example, using 
		similar reasoning as in \cite[Proposition 8.1]{gs22},
		that $\mc Q_L\times {\rm Pic}^m(C)$ is locally factorial. 
		It easily follows that the line bundle $\mc R$ extends 
		uniquely to a line bundle on $\mc Q_L\times {\rm Pic}^m(C)$, which 
		we continue to denote by $\mc R$.
		
		It is also easily seen, using Remark \ref{pull-push}, that the restriction of 
		$p_{23*}(\mc F\otimes p_{13}^*\mc P)$ to $\mc Q^s_L \times [M]$
		equals $p_{2*}(\mc F\otimes p_C^*M)$. Thus, it easily follows
		that $\mc R$ restricted to $\mc Q_L\times [M]$ equals $\beta_M$.
		Similarly, it follows that $\mc R$ restricted to $\mc Q_L\times [M']$ equals $\beta_{M'}$.
		It easily follows that $\beta_M$ is numerically equivalent 
		to $\beta_{M'}$. This completes the proof of the Lemma.
	\end{proof}
	
	\begin{remark}\label{beta_m}
		In view of the above Lemma, when $m$ satisfies \eqref{condition deg m},
		we shall denote the corresponding numerical class by $\beta_m$. 
	\end{remark}
	
	It is easily checked that when $d\gg0$ then $m=d+g-1$ satisfies 
	\eqref{condition deg m}.
	\begin{proposition}\label{beta nef}
		The class $\beta_{d+g-1}$ is globally generated and hence nef. This class is not ample. 
	\end{proposition}
	\begin{proof}
		We will show that for any point $q\in \mc Q_L$, there 
		is a line bundle $M$ on $C$ of degree $d+g-1$, such 
		that the line bundle $\beta_M$ on $\mc Q_L$ has a global section
		which does not vanish at $q$. This will show that $\beta_{d+g-1}$
		is globally generated. As a globally generated line bundle is 
		nef, it follows that $\beta_{d+g-1}$ is nef. 
		
		Consider the action of $\mb C^*$ on $\mb C^r$ given by 
		$t\cdot (a_1,\ldots,a_r)=(a_1,ta_2,\ldots,t^{r-1}a_r)$.
		This action gives rise to an action of $\mb C^*$ on $\mc O_C^{\oplus r}$
		and so also on $\mc Q_L$. Indeed, this action sends an
		inclusion $\varphi$ to the inclusion $\varphi\circ t^{-1}$, 
		in the following commutative diagram 
		\[\xymatrix{
			K\ar[r]^{\varphi\circ t^{-1}}\ar@{=}[d] & \mc O_C^{\oplus r}\ar[r]\ar[d]^t & F'\ar[d]\\
			K\ar[r]^\varphi & \mc O_C^{\oplus r}\ar[r] & F
		}
		\]
		Thus, given any point $q\in \mc Q_L$, we may find a $\mb C^*$ equivariant 
		morphism $h\colon \mb C^*\to \mc Q_L$ such that $h(1)=q$. Note that for $t\in \C^*$, the kernel 
		sheaf in $h(t)$ is the same as the kernel sheaf in $q$. The morphism $h$
		extends to a morphism $\mb C\to \mc Q_L$ and the point $h(0)$ is fixed
		under the action of $\mb C^*$ on $\mc Q_L$. Thus, $h(0)$ is a quotient $q_0$
		whose kernel equals 
		$$K_0=\bigoplus_{i=1}^{r-k}\mc O_C(-D_i)\,,$$
		where each $D_i$ is an effective divisor of degree $d_i$, and the $d_i$ satisfy 
		$\sum_i d_i=d$. See, for example, \cite[\S3]{BGL}. Let $M$ be a general line bundle of degree $d+g-1$. 
		Then $M\otimes \mc O_C(-D_i)$ is a general line bundle of degree 
		$d-d_i+g-1\geqslant g-1$. In particular, $h^1(C,M\otimes \mc O_C(-D_i))=0$. 
		It follows that $h^1(C,M\otimes K_0)=0$ and so this vanishing holds
		for the kernels in an open set containing $q_0$. In particular, it follows
		that 
		\begin{equation}\label{h^1 vanishing general M}
			h^1(C,M\otimes K)=0\,,
		\end{equation}
		where $K$ is the kernel of the quotient $q$ 
		we started with. 
		
		Since $d\gg0$, we have $d+g-1>2g-2$ and so $h^1(C,M)=0$. 
		If $q'$ is any point in $\mc Q_L$, then the cohomology long 
		exact sequence of the short exact sequence 
		$$0\to \mc K_{q'}\otimes M \to (\mc O_C^{\oplus r})\otimes M\to \mc F_{q'}\otimes M\to 0$$
		shows that $h^1(C,\mc F_{q'}\otimes M)=0$. In particular, it follows
		that the sheaf $p_{2*}(\mc F\otimes p_C^*M)$ is locally free 
		on all of $\mc Q_L$. Applying $p_{2*}(-\otimes p_C^*M)$ to \eqref{universal seq},
		and using \eqref{h^1 vanishing general M},
		it follows that on an open set containing the point $q$,
		the map 
		$$p_{2*}(p_C^*(\mc O_C^{\oplus r})\otimes p_C^*M)\to p_{2*}(\mc F\otimes p_C^*M)$$
		is a surjection. Applying $\wedge^k$ we get that the map
		$$\wedge^k\Big(H^0(C,M)^{\oplus r}\Big)\to \beta_M$$
		is surjective on an open set containing $q$. 
		It follows that $\beta_{d+g-1}$ is globally generated and so nef.
		
		To show that $\beta_{d+g-1}$ is not ample, we will construct 
		a curve $D_2\subset \mc Q_L$ such that $[D_2]\cdot [\beta_{d+g-1}]=0$. 
		Let $D_2$ be a line in $\mb P(H^0(C,L)^\vee)$. Then on $C\times D_2$
		we have a short exact sequence 
		$$0\to p_C^*L^{-1}\otimes p_2^*\mc O_{D_2}(-1)\to p_C^*\mc O_C\to \mc G_2\to 0\,.$$
		``Adding" to this, identity maps of the type  
		$p_C^*\mc O_C^{\oplus l}\to p_C^*\mc O_C^{\oplus l}$,
		we get the quotient 
		\begin{align}\label{quotient on D_2}
			0\to \Big(p_C^*L^{-1}\otimes & p_2^*\mc O_{D_2}(-1)\Big)\bigoplus 
			\Big(p_C^*\mc O_C^{\oplus (r-k-1)}\Big)\\
			&\to p_C^*\mc O_C\bigoplus p_C^*\mc O_C^{\oplus (r-k-1)}
			\bigoplus \Big(p_C^*\mc O_C^{\oplus k}\Big)\to 
			\mc G_2\bigoplus \Big(p_C^*\mc O_C^{\oplus k}\Big)\to 0\,.\nonumber
		\end{align}
		This defines a morphism 
		$$f'\colon D_2\to \mc Q_L\,,$$ 
		which is clearly injective
		on closed points. Let $M$ be a line bundle of degree $d+g-1$ 
		such that $h^1(C,L^{-1}\otimes M)=0$. Using Riemann-Roch we get that 
		$h^0(C,L^{-1}\otimes M)=0$.
		Again, as $d\gg0$, we have $h^1(C,M)=0$.
		It easily follows that for each $t\in D_2$,  
		$$h^1(C, \Big[\mc G_{2,t}\bigoplus \Big(\mc O_C^{\oplus k}\Big)\Big]\otimes M )=0$$
		It follows easily using Remark \ref{pull-push}
		that $f'^*\beta_M={\rm det}(p_{2*}((\mc G_2\bigoplus p_C^*\mc O_C^{\oplus k})\otimes p_C^*M))$.
		We claim that $p_{2*}((\mc G_2\bigoplus p_C^*\mc O_C^{\oplus k})\otimes p_C^*M)$ is the trivial 
		bundle. To see this, note that 
		$$p_{2*}\Big[p_C^*L^{-1}\otimes p_2^*\mc O_{D_2}(-1)\otimes p_C^*M\Big]=0\,,$$
		and 
		$$R^1p_{2*}\Big[\Big\{\Big(p_C^*L^{-1}\otimes p_2^*\mc O_{D_2}(-1)\Big)\bigoplus 
		\Big(p_C^*\mc O_C^{\oplus (r-k-1)}\Big)\Big\}\otimes p_C^*M\Big]=0\,.$$
		It follows that when we apply $p_{2*}(-\otimes p_C^*M)$ to the sequence 
		\eqref{quotient on D_2}, we get the following exact sequence in which the first two terms are 
		trivial bundles, 
		$$0\to H^0(C,M)^{\oplus (r-k-1)}\otimes \mc O_{D_2}\to H^0(C,M)^{\oplus r}\otimes \mc O_{D_2} 
		\to p_{2*}((\mc G_2\bigoplus p_C^*\mc O_C^{\oplus k})\otimes p_C^*M) \to 0\,.$$
		Thus, it follows that the last bundle is also trivial, and so $f'^*\beta_{M}$ is the trivial bundle.
		This shows that  $[D_2]\cdot [\beta_{d+g-1}]=0$. 
	\end{proof}

	\begin{lemma}\label{D_2 dot alpha}
		We have $[D_2]\cdot [\alpha]=[D_1]\cdot [\beta_{d+g-1}]=1$.
	\end{lemma}
	\begin{proof}
		To compute $[D_2]\cdot [\alpha]$ we shall use the description of the map $f$ 
		in the proof of Proposition \ref{alpha nef}. The kernel in the family of sheaves 
		defining $D_2$ (see \eqref{quotient on D_2}) is 
		$$\Big(p_C^*L^{-1}\otimes p_2^*\mc O_{D_2}(-1)\Big)\bigoplus 
		\Big(p_C^*\mc O_C^{\oplus (r-k-1)}\Big)\,.$$
		Taking $\wedge^{r-k}$ of this sheaf gives the line bundle $p_C^*L^{-1}\otimes p_2^*\mc O_{D_2}(-1)$.
		Thus, the pullback of $\alpha$ to $D_2$ is $\mc O_{D_2}(1)$. The degree 
		of this line bundle is 1. Thus, $[D_2]\cdot [\alpha]=1$.
		
		Recall the family of quotients \eqref{family D_1} which defines a morphism
		$D_1\to \mc Q_L$. Using cohomology and base change, it easily follows 
		that the pullback of $\beta_{d+g-1}$ to $D_1$ is ${\rm det}(p_{2*}(\mc G_1\otimes M))$,
		where $M$ is any line bundle of degree $d+g-1$.
		From the short exact sequence (see \eqref{family D_1})
		$$0\to \iota_{x*}(\mc O_{D_1}(1)) \to \mc G_1\to p_C^*F'\to 0$$
		it easily follows that ${\rm det}(p_{2*}(\mc G_1\otimes M))=\mc O_{D_1}(1)$.
		Thus, $[D_1]\cdot [\beta_{d+g-1}]=1$.
	\end{proof}
	
	\begin{remark}
		As a Corollary of $[D_2]\cdot [\alpha]=1$, we see that $\alpha$ and $\beta_{d+g-1}$ 
		are not numerically equivalent. Thus, it follows that the natural map 
		${\rm Pic}(\mc Q_L)\to {\rm N}^1(\mc Q_L)$ is an isomorphism. 
	\end{remark}

	Putting together the above results, we have the following Theorem, which is 
	part \ref{intro 1} of Theorem \ref{thm intro}.
	
	\begin{theorem}\label{main theorem nef}
		Assume one of the following two holds:
		\begin{itemize}
			\item $g(C)\geqslant 3$ and $2\leqslant k\leqslant r-2$, or
			\item $g(C)=2$ and $3\leqslant k\leqslant r-2$.
		\end{itemize}Let $d\gg0$. Then 
		${\rm Pic}(\mc Q_L)$ is generated by the line bundles
		$\alpha$, $\beta_{d+g-1}$. Both these are globally generated and so nef, 
		but not ample.
		In particular, they are the boundaries of the nef cone. 
		The boundaries of the cone of effective curves are given by the classes 
		of $D_1$ and $D_2$. 
	\end{theorem}
	\begin{proof}
		Given an integral curve $C'\subset \mc Q_L$, let us write 
		$[C']=a[D_1]+b[D_2]$ in $N_1(\mc Q_L)$. Intersecting with 
		$\alpha$ and $\beta_{d+g-1}$ we easily see that $a\geqslant 0$ 
		and $b\geqslant 0$. 
	\end{proof}

	\section{Effective cone}
	In this section we shall determine the cone of effective divisors in 
	the Picard group of $\mc Q_L$. Let 
	\begin{equation}\label{def Q'}
		\mc Q'_L=\{[\mc O_C^{\oplus r}\xrightarrow{q} F]\in \mc Q_L^s\,\vert\, 
		{\rm Ker}(q)\text{ is stable} \}\,. 
	\end{equation} 
	denote the open set consisting 
	of quotients such that both $F$ and ${\rm Ker}(q)$ 
	are stable. 
	
	\begin{lemma}\label{codim Q'}
		We have ${\rm codim}(\mc Q_L\setminus \mc Q'_L,\mc Q_L)\geqslant 2$.
	\end{lemma}
	\begin{proof}
		Let $\mc Q^{\rm tf}_L\subset \mc Q_L$ denote the locus of 
		quotients $[\mc O_C^{\oplus r}\to F]$ such that $F$ is torsion 
		free. Note that $\mc Q^s_L\subset \mc Q^{\rm tf}_L$. 
		For ease of notation, let us denote the Quot scheme 
		${\rm Quot}_{C/\C}(\mc O_C^{\oplus r},r-k,d)$ by $\widetilde{\mc Q}$. 
		There is an isomorphism of schemes 
		$\varphi\colon  \mc Q^{\rm tf}_L\to \widetilde{\mc Q}^{\rm tf}_L$
		which sends a quotient 
		$$[{\rm Ker}(q)\subset \mc O_C^{\oplus r}\xrightarrow{q}F]\mapsto 
		[F^{\vee}\subset  \mc O_C^{\oplus r}\to {\rm Ker}(q)^\vee]\,.$$
		The open set $\mc Q'_L$ is precisely the intersection 
		$\mc Q^s_L\cap \varphi^{-1}(\widetilde{\mc Q}^s_L)$. 
		By the discussion before equation (8.10)
		in \cite{gs22}, it follows that 
		${\rm codim}(\mc Q_L\setminus \mc Q^s_L,\mc Q_L)\geqslant 2$. Similarly, 
		${\rm codim}(\widetilde{\mc Q}^{\rm tf}_L\setminus \widetilde{\mc Q}^s_L, \widetilde{\mc Q}^{\rm tf}_L)\geqslant 2$.
		Thus, it follows that the codimension of 
		$\mc Q^{\rm tf}_L\setminus\varphi^{-1}(\widetilde{\mc Q}^s_L)$
		in $\mc  Q^{\rm tf}_L$ is at least 2. The Lemma now follows easily.
	\end{proof}

	We will use the above Lemma to write down another curve $D_3\to \mc Q_L$ such 
	that the image lies in $\mc Q_L'$. Let $E$ be a stable bundle  of rank $r-k$
	with ${\rm det}(E)=L^{-1}$. Consider the space $\mb P({\rm Hom}(E,\mc O_C^{\oplus r})^\vee)$.
	The closed points of this space are in bijection with nonzero 
	maps $\mc O_C^{\oplus r}\to E^\vee$. 
	The locus of points in $\mb P({\rm Hom}(E,\mc O_C^{\oplus r})^\vee)$ 
	corresponding to non-surjective
	maps $\mc O_C^{\oplus r}\to E^\vee$ has codimension at least 2, see the 
	proof in \cite[Lemma 7.12]{gs22}. Let 
	$$U_{E^\vee}\subset \mb P({\rm Hom}(\mc O_C^{\oplus r},E^\vee)^\vee)=
	\mb P({\rm Hom}(E,\mc O_C^{\oplus r})^\vee)$$
	denote the locus parametrizing surjective maps $\mc O_C^{\oplus r}\to E^\vee$. Let $M^s_{r-k,L}$ 
	denote the moduli space of stable bundles of rank $r-k$ and determinant $L$. 
	Consider the natural map 
	\begin{equation}\label{map to moduli space}
		\pi\colon \widetilde{\mc Q}^s_L\to M^s_{r-k,L}\,,
	\end{equation}
	which sends 
	$[\mc O_C^{\oplus r}\to F]\mapsto [F]$. The fiber over the point $[F]$ 
	is precisely the set $U_F$.
	Let $T$ denote the closed subset 
	$\widetilde{\mc Q}^s_L\setminus \widetilde{\mc Q}_L'$.
	
	\begin{lemma}
		For general $F\in M^s_{r-k,L}$, ${\rm codim}(T\cap U_F,U_F)\geqslant 2$.
	\end{lemma}
	\begin{proof}
		If ${\rm codim}(T\cap U_F,U_F)\leqslant  1$ for general $[F]\in M^s_{r-k,L}$,
		then it follows that 
		$${\rm dim}(T)={\rm dim}(M^s_{r-k,L})+{\rm dim}(\pi^{-1}([F]))-1={\rm dim}(\widetilde{\mc Q}_L^s)-1\,.$$
		This contradicts Lemma \ref{codim Q'}, which says that 
		${\rm codim}(T,\widetilde{\mc Q}^s_L)\geqslant 2$.
	\end{proof}
	Thus, it follows that for general $E$, the locus of points 
	in $U_{E^\vee}$, such that the kernel of $\mc O_C^{\oplus r}\to E^\vee$ is not stable,
	has codimension $\geqslant 2$. In other words, if 
	$U'\subset \mb P({\rm Hom}(E,\mc O_C^{\oplus r})^\vee)$
	denotes the set of points corresponding to inclusions
	$E \to \mc O_C^{\oplus r}$ such that the cokernel is torsion
	free and stable, then for general $E$,
	$${\rm codim}(\mb P({\rm Hom}(E,\mc O_C^{\oplus r})^\vee)\setminus U',
	\mb P({\rm Hom}(E,\mc O_C^{\oplus r})^\vee))\geqslant 2\,.$$ 
	
	If $W\subset \mb P^n$ is a closed 
	subset such that ${\rm codim}(W,\mb P^n)\geqslant 2$, then the general  
	line in $\mb P^n$ does not meet $W$. This is easily seen by projecting 
	from a point outside $W$. 
	Thus, we can find a line $D_3\subset \mb P({\rm Hom}(E,\mc O_C^{\oplus r})^\vee)$,
	which is completely contained in $U'$. 
	We get a family of quotients 
	on $C\times D_3$
	\begin{equation}\label{family D_3}
		0\to p_C^*E\otimes p_2^*\mc O_{D_3}(-1)\to  p_C^*\mc O_C^{\oplus r}\to \mc G_3\to 0\,,
	\end{equation}
	such that for each $t\in D_3$, the sheaf $\mc G_{3,t}$ is stable. 
	The above family defines a morphism $D_3\to \mc Q_L$, which 
	is injective on closed points. Clearly, the image of $D_3$ 
	lands in $\mc Q_L'$.

	\begin{lemma}\label{alpha dot D_3}
		We have $[\alpha]\cdot [D_3]=r-k$ and $[\beta_{d+g-1}]\cdot [D_3]=d(r-k-1)$.
	\end{lemma}
	\begin{proof}
		To compute $[\alpha]\cdot [D_3]$ we follow the description 
		of the map $f$ (see \eqref{def f}) given in Proposition \ref{alpha nef}. 
		$$\wedge^{r-k}(p_C^*E\otimes p_2^*\mc O_{D_3}(-1))=p_C^*L^{-1}\otimes p_2^*\mc O_{D_3}(-(r-k))\,.$$
		Thus, it follows that $[\alpha]\cdot [D_3]=r-k$. 
		
		Let $M$ be a line bundle of degree $d+g-1$. By Serre duality, we have, 
		$h^1(C,E\otimes M)=h^0(E^\vee\otimes M^\vee\otimes \omega_C)$.
		Note $E$ is stable and 
		$$\mu(E^\vee\otimes M^\vee\otimes \omega_C)=\frac{d}{r-k}-(d+g-1)+2g-2\,.$$
		The slope is $<0$ for $d\gg0$ as $r-k\geqslant 2$. Thus, $h^0(E^\vee\otimes M^\vee\otimes \omega_C)=h^1(C,E\otimes M)=0$.
		Thus, applying $p_{2*}(-\otimes p_C^*M)$ to \eqref{family D_3}, we get 
		the short exact sequence 
		$$0\to H^0(C,E\otimes M)\otimes \mc O_{D_3}(-1)\to H^0(C,M)^{\oplus r}\otimes \mc O_{D_3}\to 
		p_{2*}(\mc G_3\otimes p_C^*M)\to 0\,.$$
		It follows that 
		\begin{align*}
			[D_3]\cdot [\beta_{d+g-1}] & = {\rm deg}(p_{2*}(\mc G_3\otimes p_C^*M))\\
			&=h^0(C,E\otimes M)=\chi(E\otimes M)\\
			&=d(r-k-1)\,.
		\end{align*}
		This completes the proof of the Lemma.
	\end{proof}

	Consider the map $\pi\colon \mc Q^s_L\to M^s_{k,L}$,
	where $M^s_{k,L}$ denotes the moduli space parametrizing 
	stable bundles of rank $k$ and determinant $L$. The map
	$\pi$ sends a quotient $[\mc O_C^{\oplus r}\to F]$ to $[F]$.
	The fiber of $\pi$ over the point $[F]$ is the subset 
	$U_F\subset \mb P({\rm Hom}(\mc O_C^{\oplus r},F)^\vee)$
	corresponding to surjective maps. Let $U_F'\subset U_F$
	be the subset corresponding to maps for which the kernel 
	is also a stable bundle. 
	Arguing as in the construction of the curve $D_3$, we see 
	that for a general stable bundle $F$, one has 
	$${\rm codim}(\mb P({\rm Hom}(\mc O_C^{\oplus r},F)^\vee)\setminus U_F',
	\mb P({\rm Hom}(\mc O_C^{\oplus r},F)^\vee))\geqslant 2\,.$$
	Thus, taking $F$ general stable and taking $D_4$ to be a general line in 
	$\mb P({\rm Hom}(\mc O_C^{\oplus r},F)^\vee)$, to get a 
	family 
	\begin{equation}\label{family D_4}
		0\to \mc K_4\to  p_C^*\mc O_C^{\oplus r}\to p_C^*F\otimes p_2^*\mc O_{D_4}(1)\to 0
	\end{equation}
	on $C\times D_4$. Again, this family has the property that for each $t\in D_4$,
	the sheaf $\mc K_{4,t}$ is stable. 
	
	\begin{lemma}\label{alpha dot D_4}
		We have $[\alpha]\cdot [D_4]=k$ and $[\beta_{d+g-1}]\cdot [D_4]=d(k+1)$.
	\end{lemma}
	\begin{proof}
		Note that $\wedge^{r-k}\mc K_4=(p_C^*L\otimes p_2^*\mc O_{D_4}(k))^{-1}$.
		Using the description in Proposition \ref{alpha nef}, it follows
		that $[\alpha]\cdot [D_4]=k$. Let $M$ be a line bundle of degree 
		$d+g-1$. Since  
		$$p_{2*}(p_C^*(F\otimes M)\otimes p_2^*\mc O_{D_4}(1))=H^0(C,F\otimes M)\otimes \mc O_{D_4}(1)\,,$$
		and $H^1(C,F\otimes M)=0$, it follows that 
		$[\beta_{d+g-1}]\cdot [D_4]=\chi(F\otimes M)=d(k+1)$.
	\end{proof}
	
	\begin{lemma}\label{half plane effective}
		Let $a$ and $b$ be integers such that $a\alpha+b\beta_{d+g-1}$ 
		is an effective divisor. Then $ak+bd(k+1)\geqslant 0$. 
	\end{lemma}
	\begin{proof}
		Let $Y\subset \mc Q_L$ be an effective divisor. 
		Then $Y$ cannot contain all the fibers of the map 
		$\pi\colon \mc Q^s_L\to M^s_{k,L}$. Thus, for general $F$,
		the intersection $Y\cap U_{F}'\subsetneqq U'_{F}$.
		In particular, $Y$ does not contain the general line 
		in $U'_{F}$, that is, $Y$ does not contain $D_4$. 
		Thus, $[Y]\cdot [D_4]\geqslant 0$. Letting 
		the class of $Y$ to be $a\alpha+b\beta_{d+g-1}$, the Lemma 
		follows easily. 
	\end{proof}
	
	Let ${\rm Pic}(M_{k,L})={\rm Pic}(M_{k,L}^s)=\Z[\Theta]$, where 
	$\Theta$ is the unique ample generator. 
	\begin{lemma}\label{computing theta class}
		Let $\lambda_0:={\rm gcd}(k,d(k+1))={\rm gcd}(k,d)$. Then
		$\pi^*\Theta=\frac{1}{\lambda_0}(d(k+1)\alpha-k\beta_{d+g-1})$.
	\end{lemma}
	\begin{proof}
		Let us write $\pi^*\Theta=a\alpha+b\beta_{d+g-1}$. Clearly,
		$[\pi^*\Theta]\cdot [D_4]=0$ as $\pi(D_4)=[F]$.
		This gives 
		$$ak+bd(k+1)=0\,.$$
		Thus, $\pi^*\Theta=\lambda'(-d(k+1)\alpha + k\beta_{d+g-1})$, for some 
		rational number $\lambda'$. We claim that $\lambda'\neq 0$. This is clear 
		from the point \eqref{Picard rank 2}, which says that the pullback of $\Theta$
		is not trivial. 
		
		Note that $[\pi(D_3)]\cdot [\Theta]\geqslant 0$ and so we get 
		$[\pi^*\Theta]\cdot [D_3]\geqslant 0$.  Using Lemma \ref{alpha dot D_3}
		we get 
		\begin{align*}
			[\pi^*\Theta]\cdot [D_3]&=\lambda'(-d(k+1)\alpha + k\beta_{d+g-1})\cdot [D_3]\\
			&=\lambda'(-d(k+1)(r-k)+kd(r-k-1))\\
			&=\lambda'(-dr)
		\end{align*}
		The condition that $[\pi^*\Theta]\cdot [D_3]\geqslant 0$ now forces 
		that $\lambda'<0$. Thus, we get that $\pi^*\Theta=\lambda(d(k+1)\alpha- k\beta_{d+g-1})$,
		where $\lambda>0$ is a rational number.
		
		We can determine the precise value of $\lambda$ as follows.
		First we need to recall some facts from \cite{gs22}.
		In the proof of \cite[Theorem 8.9]{gs22},  
		it is shown that there is a commutative diagram 
		\[\xymatrix{
			1\ar[r]& {\rm Pic}(M^s_{k,L\otimes \mc O_C(knP)})\ar[r]\ar@{=}[d] & {\rm Pic}(\mc Q^s_L)\ar[r]\ar@{=}[d] & \Z\ar[r]\ar@{^(->}[d] & 1\phantom{\,.}\\
			1\ar[r]& {\rm Pic}(M^s_{k,L\otimes \mc O_C(knP)})\ar[r] & {\rm Pic}(\mc Q^s_L)\ar[r] & {\rm Pic}(\theta^{-1}[F\otimes \mc O_C(nP)])\cong \Z
		}
		\]
		in which the rows are exact sequences. 
		We recall the map $\theta\colon \mc Q^s_L\to M^s_{k,L\otimes \mc O_C(knP)}$
		is defined as $\theta([\mc O_C^{\oplus r}\to F])=[F\otimes \mc O_C(nP)]$.
		The existence of the above diagram is proved using the same method as 
		described in the second paragraph of the proof 
		of \cite[Theorem 7.17]{gs22}. Fix a point $P\in C$
		and let $n\gg0$ be a fixed integer, as in the discussion 
		in the beginning of \S7 in \cite{gs22}. 
		Consider the isomorphism
		$\delta\colon  M^s_{k,L}\to M^s_{k,L\otimes \mc O_C(knP)}$ 
		given by $[F]\mapsto [F\otimes \mc O_C(nP)]$. Then $\theta=\delta\circ \pi$
		and so $\theta^{-1}([F\otimes \mc O_C(nP)])=\pi^{-1}([F])=U_F$.
		Thus, the lower row in the above diagram is identified with 
		$$1\to {\rm Pic}(M^s_{k,L})\to {\rm Pic}(\mc Q^s_L)\to {\rm Pic}(U_F)\cong \Z\,.$$
		Given an element $\gamma\in {\rm Pic}(\mc Q^s_L)$, its image 
		in ${\rm Pic}(U_F)\cong \Z$ is the intersection of a general
		line in $U_F$ with $\gamma$, that is, $[\gamma]\cdot [D_4]$.
		Thus, to compute $\pi^*\Theta$ in terms of $\alpha$ and $\beta_{d+g-1}$,
		we need to compute the generator of the kernel 
		of the map $\Z^{\oplus 2}\xrightarrow{\varphi} \Z$
		which sends $\varphi(1,0)=k$ and $\varphi(0,1)=d(k+1)$. 
		Let $\lambda_0:={\rm gcd}(k,d(k+1))={\rm gcd}(k,d)$. It is easily 
		checked that this generator, that is, $\pi^*\Theta$, equals 
		$$\pi^*\Theta=\frac{d(k+1)}{\lambda_0}\alpha-\frac{k}{\lambda_0}\beta\,.$$
		This completes the proof of the Lemma. 
	\end{proof}
	
	\begin{remark}
		As a corollary, we also get the following. 
		Since ${\rm Pic}(\mc Q^s_L)$ is generated by $\alpha$ 
		and $\beta_{d+g-1}$, it follows that the image of the restriction map 
		${\rm Pic}(\mc Q^s_L)\to {\rm Pic}(U_F)$ is generated 
		by ${\rm gcd}(k,d(k+1))={\rm gcd}(k,d)$. \hfill\qedsymbol
	\end{remark}
	
	As a corollary of Lemma \ref{half plane effective} and Lemma \ref{computing theta class}
	we get the following Corollary. 
	
	\begin{corollary}\label{boundary effective 1}
		The class $\pi^*\Theta$ is a boundary of the effective cone. 
	\end{corollary}
	\begin{proof}
		It is clear that $\pi^*\Theta$ is an effective divisor. 
		By Lemma \ref{half plane effective}, if $a\alpha+b\beta_{d+g-1}$
		is effective, then $a$ and $b$ satisfy the inequality 
		$ak+bd(k+1)\geqslant 0$. It is clear that the coefficients 
		$a$ and $b$ in $\pi^*\Theta$ satisfy the equality 
		$ak+bd(k+1)= 0$. Thus, the Corollary follows. 
	\end{proof}
	
	Recall the space $\mc Q'_L$ defined before Lemma \ref{codim Q'}.
	On $\mc Q'_L$ we have the map $\pi'\colon \mc Q'_L\to M^s_{r-k,L}$,
	which sends a quotient $[\mc O_C^{\oplus r}\xrightarrow{q}F]$
	to $[{\rm Ker}(q)]$. Arguing as in Lemma \ref{half plane effective},
	we have the following Lemma.

	\begin{lemma}\label{half plane effective-2}
		Let $a$ and $b$ be integers such that $a\alpha+b\beta_{d+g-1}$ 
		is an effective divisor. Then $a(r-k)+bd(r-k-1)\geqslant 0$. 
	\end{lemma}
	\begin{proof}
		The proof is the same as that of Lemma \ref{half plane effective},
		except that we use the map $\pi'$ now. In this case we will have 
		the condition $[Y]\cdot [D_3]\geqslant 0$. The Lemma easily follows
		using Lemma \ref{alpha dot D_3}.
	\end{proof}
	
	Let ${\rm Pic}(M_{r-k,L})={\rm Pic}(M_{r-k,L}^s)=\Z[\Theta']$, where 
	$\Theta'$ is the unique ample generator. Similar to Lemma \ref{computing theta class},
	we have the following Lemma. 
	
	\begin{lemma}\label{computing theta class-2}
		Let $\lambda_1:={\rm gcd}(r-k,d(r-k-1))={\rm gcd}(r-k,d)$. Then
		$$\pi'^*\Theta'=\frac{1}{\lambda_1}(d(r-k-1)\alpha-(r-k)\beta_{d+g-1})\,.$$
	\end{lemma}
	\begin{proof}
		Let us write $\pi'^*\Theta'=a\alpha+b\beta_{d+g-1}$. 
		Recall the curve $D_3$ defined using the family \eqref{family D_3}.
		Clearly, $[\pi'^*\Theta']\cdot [D_3]=0$ as $\pi'(D_3)=[E]$.
		This gives 
		$$a(r-k)+bd(r-k-1)=0\,.$$
		Thus, $\pi'^*\Theta'=\lambda'(-d(r-k-1)\alpha + (r-k)\beta_{d+g-1})$, for some 
		rational number $\lambda'$. As in Lemma \ref{computing theta class}, 
		we have that $\lambda'\neq0$. 
		Note that $[\pi'(D_4)]\cdot [\Theta']\geqslant 0$ and so we get 
		$[\pi'^*\Theta']\cdot [D_4]\geqslant 0$.  Using Lemma \ref{alpha dot D_4}
		we get 
		\begin{align*}
			[\pi'^*\Theta']\cdot [D_4]&=\lambda'(-d(r-k-1)\alpha + (r-k)\beta_{d+g-1})\cdot [D_4]\\
			&=\lambda'(-d(r-k-1)k+(r-k)d(k+1))\\
			&=\lambda'(dr)
		\end{align*}
		The condition that $[\pi'^*\Theta']\cdot [D_4]\geqslant 0$ now forces 
		that $\lambda'>0$. Thus, we get that $\pi'^*\Theta'=\lambda(-d(r-k-1)\alpha + (r-k)\beta_{d+g-1})$,
		where $\lambda>0$ is a rational number.
		Arguing as in Lemma \ref{computing theta class}, we get 
		$$\pi'^*\Theta'=\frac{-d(r-k-1)}{\lambda_1}\alpha+\frac{r-k}{\lambda_1}\beta_{d+g-1}\,.$$
		This completes the proof of the Lemma.
	\end{proof}
	
	As a corollary of Lemma \ref{half plane effective-2} and Lemma \ref{computing theta class-2}
	we get the following Corollary. 
	
	\begin{corollary}
		The class $\pi'^*\Theta'$ is a boundary of the effective cone. 
	\end{corollary}
	
	Thus, combining the above results, we have the following Theorem, which is 
	part \ref{intro 2} of Theorem \ref{thm intro}
	
	\begin{theorem}\label{main theorem effective cone}
		Assume one of the following two holds:
		\begin{itemize}
			\item $g(C)\geqslant 3$ and $2\leqslant k\leqslant r-2$, or
			\item $g(C)=2$ and $3\leqslant k\leqslant r-2$.
		\end{itemize}
		Let $d\gg0$. 
		The effective cone of $\mc Q_L$ is spanned by non-negative
		linear combinations of the classes
		$d(k+1)\alpha-k\beta_{d+g-1}$ and $-d(r-k-1)\alpha+(r-k)\beta_{d+g-1}$.
		Further, ${\rm Mov}(\mc Q_L)={\rm Eff}(\mc Q_L)$.
	\end{theorem}
	\begin{proof}
		Clearly, ${\rm Mov}(\mc Q_L)\subset {\rm Eff}(\mc Q_L)$. Since the boundaries 
		of ${\rm Eff}(\mc Q_L)$, namely $\pi^*\Theta$ and $\pi'^*\Theta'$, 
		define morphisms on the open subset $\mc Q_L'$, whose complement in 
		$\mc Q_L$ has codimension $\geqslant2$, it follows that these boundaries
		are in ${\rm Mov}(\mc Q_L)$. Thus, equality follows. 
	\end{proof}

	\section{Canonical divisor}
	In this section we shall determine the canonical divisor of $\mc Q_L$
	in terms of $\alpha$ and $\beta$.

	Let $\omega_C$ denote the canonical divisor of $C$.
	Consider the open subset $\mc Q_g\subset \mc Q$ 
	consisting of quotients $[\mc O_C^{\oplus r}\to F]$
	for which $h^1(C,F)=0$. If $h^1(C,F)=0$, then applying ${\rm Hom}(-,F)$
	to the short exact sequence $0\to K\to \mc O_C^{\oplus r}\to F\to 0$,
	it follows that ${\rm ext}^1(K,F)=0$. Thus, $\mc Q_g$ is contained in the 
	smooth locus of $\mc Q$. Recall that $\mc Q^s$ denoted the open subset 
	consisting of quotients $[\mc O_C^{\oplus r}\to F]$ for which $F$ is stable.
	Clearly, $\mc Q^s\subset \mc Q_g$ as $d\gg0$. Using Lemma 2.7 and equation (6.4)
	in \cite{gs22}, it follows that the morphism ${\rm det}\colon \mc Q_g\to {\rm Pic}^d(C)$ 
	is a smooth morphism. It follows that the locus $\mc Q_{g,L}=\mc Q_L\cap \mc Q_g$
	is contained in the smooth locus of $\mc Q_L$. 
	As ${\rm codim}(\mc Q_L\setminus \mc Q_L^s,\mc Q_L)\geqslant 2$, it follows
	that ${\rm codim}(\mc Q_L\setminus \mc Q_{g,L},\mc Q_L)\geqslant 2$.
	Thus, to determine the canonical divisor of $\mc Q_L$, it suffices to determine
	the canonical divisor of $\mc Q_{g,L}$. As the morphism ${\rm det}$ is smooth on $\mc Q_g$,
	and the canonical divisor of ${\rm Pic}^d(C)$ is trivial, it follows easily from
	the exact sequence (det being the morphism in \eqref{def det})
	$$0\to {\rm det}^*\Omega_{{\rm Pic}^d(C)}\to \Omega_{\mc Q_g}\to \Omega_{\rm det}\to 0$$
	that 
	\begin{equation}
		{\rm det}(\Omega_{\mc Q_{g,L}})={\rm det}(\Omega_{\rm det}\vert_{\mc Q_{g,L}})={\rm det}(\Omega_{\mc Q_g})\vert_{\mc Q_{g,L}}\,.
	\end{equation}
	
	Recall the universal sequence \eqref{universal seq} on $C\times \mc Q$.
	Using the same method as in \cite[Theorem 7.1]{Str}, we may show 
	that the tangent bundle on $\mc Q_g$ equals $p_{2*}(\mc K^\vee\otimes \mc F)$.
	It easily follows that 
	\begin{equation}\label{a4}
		\Omega_{\mc Q_g}\vert_{\mc Q_{g,L}}=(p_{2*}(\mc K^\vee\otimes \mc F))^\vee\,,
	\end{equation}
	where we use the same notation to denote the restriction of the sheaves
	$\mc K, \mc F$ to $C\times \mc Q_{g,L}$.
	Thus, the canonical divisor of $\mc Q_{g,L}$ equals the determinant of the locally 
	free sheaf $(p_{2*}(\mc K^\vee\otimes \mc F))^\vee$.
	
	To compute the canonical divisor in terms of $\alpha$ and $\beta$,
	we need two curves in $\mc Q_{g,L}$. One of these is the curve 
	$D_1$, given by the family \eqref{family D_1}.
	Let us check that the image of $D_1$ is contained in $\mc Q_{g,L}$. 
	Recall from \cite[Corollary 6.3]{PR} that when $d\gg0$, and $K'$ is a general stable 
	bundle, then the cokernel of the general inclusion $K'\to \mc O_C^{\oplus r}$ 
	is a stable bundle. In particular, we may assume that both $K'$ and 
	$F'$ in \eqref{family D_1} are stable bundles. For
	each $t\in D_1$, the quotient $\mc G_{1,t}\cong F'\oplus \mb C_x$
	and so 
	$$h^1(C,\mc G_{1,t})=0\,.$$
	It follows that the image of $D_1$ is contained in $\mc Q_{g,L}$. 
	
	Our second curve is the curve $D_3$ given by the family \eqref{family D_3}.
	Recall the space $\mc Q'_L$ from \eqref{def Q'}.
	We had seen that the image of $D_3\to \mc Q_L$ is contained in $\mc Q'_L$.
	Also note that $\mc Q'_L\subset \mc Q_{g,L}$. 
	Thus, the curves $D_1$ and $D_3$ are contained in $\mc Q_{g,L}$. Next we 
	will compute the degree of the line bundle ${\rm det}(p_{2*}(\mc K^\vee\otimes \mc F))$ 
	restricted to $D_1$ and $D_3$. 
	\begin{lemma}\label{omega dot D_1}
		The degree of ${\rm det}(p_{2*}(\mc K^\vee\otimes \mc F))$ 
		restricted to $D_1$ is $r-2k$. 
	\end{lemma}
	\begin{proof}
		As ${\rm ext}^1(K,F)=0$ for a point 
		$[K\subset \mc O_C^{\oplus r}\to F]\in \mc Q_g$, it follows easily that 
		(see \eqref{family D_1})
		$$p_{2*}(\mc K^\vee\otimes \mc F)\vert_{D_1}=p_{2*}(\t K_1^\vee\otimes \mc G_1)\,.$$
		From \eqref{family D_1} it follows that we have the following short exact 
		sequence on $C\times D_1$
		$$0\to \iota_{x*}(\mc O_{D_1}(1))\to \mc G_1\to p_C^*F'\to 0\,.$$
		For ease of notation, let $\mc T$ denote the sheaf $\iota_{x*}(\mc O_{D_1}(1))$. 
		Applying $p_{2*}(\t K_1^\vee\otimes -)$ we get the short exact sequence 
		\begin{equation}\label{a1}
			0\to p_{2*}(\t K_1^\vee\otimes \mc T) \to 
			p_{2*}(\t K_1^\vee\otimes \mc G_1) \to p_{2*}(\t K_1^\vee\otimes p_C^*F')\to 0\,.
		\end{equation}
		Let us first compute determinant of the sheaf 
		$p_{2*}(\t K_1^\vee\otimes \mc T)$. 
		Apply ${\ms H}om(-,\mc T)$ to the short exact sequence 
		$0\to \t K_1 \to p_C^*K'\to \mc T\to 0$ (see \eqref{family D_1}) yields 
		the long exact sequence 
		\begin{equation}\label{a5}
			0\to {\ms H}om(\mc T,\mc T)\to \mc T^{\oplus (r-k)}\to \t K_1^\vee\otimes \mc T
			\to {\ms E}xt^1(\mc T,\mc T)\to 0\,.
		\end{equation}
		Applying ${\ms H}om(-,\mc T)$ to the short exact sequence 
		\begin{equation}\label{ses Tau}
			0\to p_C^*\mc O_C(-x)\otimes p_2^*\mc O_{D_1}(1)\to p_2^*\mc O_{D_1}(1)\to \mc T\to 0\,,
		\end{equation}
		one easily checks that the sheaves ${\ms H}om(\mc T,\mc T)$ and 
		${\ms E}xt^1(\mc T,\mc T)$ are isomorphic to $\iota_{x*}(\mc O_{D_1})$.
		As all the sheaves in \eqref{a5} are coherent over $D_1$, applying $p_{2*}$ 
		we get the following exact sequence of sheaves on $D_1$
		$$0\to \mc O_{D_1}\to \mc T^{\oplus (r-k)}\to p_{2*}(\t K_1^\vee\otimes \mc T)
		\to \mc O_{D_1}\to 0\,.$$
		From this it follows that 
		\begin{equation}\label{a3}
			{\rm det}(p_{2*}(\t K_1^\vee\otimes \mc T))\cong \mc O_{D_1}(r-k)\,.
		\end{equation}
		Next let us compute the determinant of the sheaf 
		$p_{2*}(\t K_1^\vee\otimes p_C^*F')$. 
		For this we apply ${\ms H}om(-,p_C^*F')$ to the short exact sequence 
		$$0\to \t K_1\to p_C^*K'\to \mc T\to 0\,.$$
		We get the following long exact sequence on $C\times D_1$
		\begin{equation}\label{a2}
			0\to p_C^*K'^\vee\otimes p_C^*F'\to \t K_1^\vee\otimes p_C^*F'
			\to {\ms E}xt^1(\mc T,p_C^*F')\to 0\,.	
		\end{equation}
		The last term equals 
		\begin{align*}
			{\ms E}xt^1(\mc T,p_C^*F')&\cong {\ms E}xt^1(\mc T\otimes p_C^*F',\mc O_{C\times D_1})\\
			&\cong {\ms E}xt^1(\mc T,\mc O_{C\times D_1})^{\oplus k}
		\end{align*}
		Applying ${\ms H}om(-,\mc O_{C\times D_1})$ to \eqref{ses Tau} one easily sees
		that ${\ms E}xt^1(\mc T,\mc O_{C\times D_1})\cong \iota_{x*}(\mc O_{D_1}(-1))$.
		Note that $h^1(K'^\vee \otimes F')=0$ as both $K'$ and $F'$ are stable. 
		Applying $p_{2*}$ to \eqref{a2}, we get the following 
		exact sequence
		$$0\to {\rm Hom}(K',F')\to p_{2*}(\t K_1^\vee\otimes p_C^*F')\to \mc O_{D_1}(-1)^{\oplus k}\to 0\,.$$
		It follows that ${\rm det}(p_{2*}(\t K_1^\vee\otimes p_C^*F'))\cong \mc O_{D_1}(-k)$.
		Using this and equation \eqref{a3} in \eqref{a1}, we get 
		${\rm det}(p_{2*}(\t K_1^\vee\otimes \mc G_1))=\mc O_{D_1}(r-2k)$.
	\end{proof}

	\begin{lemma}\label{omega dot D_3}
		The degree of ${\rm det}(p_{2*}(\mc K^\vee\otimes \mc F))$ 
		restricted to $D_3$ is $r(d+(r-k)(1-g))$. 
	\end{lemma}
	\begin{proof}
		Recall from \eqref{family D_3} the family parameterized by $D_3$. 
		It follows that 
		$$p_{2*}(\mc K^\vee\otimes \mc F)\vert_{D_3}=
		p_{2*}(p_C^*E^\vee\otimes  p_2^*\mc O_{D_3}(1)\otimes \mc G_3)\,.$$
		Note that as $E^\vee$ is a stable bundle of degree $d\gg0$, we have 
		$H^1(C,E^\vee)=0$.
		Tensoring \eqref{family D_3} with 
		$p_C^*E^\vee\otimes p_2^*\mc O_{D_3}(1)$ and applying $p_{2*}$ 
		yields the long exact sequence 
		\begin{align*}
			0\to {\rm Hom}(E,E)\otimes \mc O_{D_3}&\to [H^0(C,E^\vee)\otimes \mc O_{D_3}(1)]^{\oplus r}\to\\ 
			&p_{2*}(p_C^*E^\vee\otimes  p_2^*\mc O_{D_3}(1)\otimes \mc G_3)\to {\rm Ext}^1(E,E)\otimes \mc O_{D_3}\to 0\,.
		\end{align*}
		It follows that 
		\begin{align*}
			{\rm det}(p_{2*}(p_C^*E^\vee\otimes  p_2^*\mc O_{D_3}(1)\otimes \mc G_3))&\cong 
			\mc O_{D_3}(rh^0(C,E^\vee))\\
			&= \mc O_{D_3}(r\chi(E^\vee))\\
			&=\mc O_{D_3}(r(d+(r-k)(1-g)))\,.
		\end{align*}
		This completes the proof of the Lemma.
	\end{proof}
	
	We have the following Theorem, which is 
	part \ref{intro 3} of Theorem \ref{thm intro}
	
	\begin{theorem}\label{main theorem canonical}
		Assume one of the following two holds:
		\begin{itemize}
			\item $g(C)\geqslant 3$ and $2\leqslant k\leqslant r-2$, or
			\item $g(C)=2$ and $3\leqslant k\leqslant r-2$.
		\end{itemize}
		Let $d\gg0$. Let $\omega_{\mc Q_L}$ denote the canonical 
		divisor of $\mc Q_L$. Then 
		$$\omega_{\mc Q_L}=[d(r-2k-2)+r(g-1)]\alpha+(2k-r)\beta_{d+g-1}\,.$$
		In particular, $\mc Q_L$ is Fano iff $r=2k+1$. 
	\end{theorem}
	\begin{proof}
		Let us write $\omega_{\mc Q_L}=a\alpha+b\beta_{d+g-1}$. 
		Recall $\omega_{\mc Q_L}={\rm det}(p_{2*}(\mc K^\vee\otimes \mc F))^\vee$.
		It follows from Lemma \ref{omega dot D_1} and Lemma \ref{omega dot D_3}
		that 
		\begin{align*}
			[\omega_{\mc Q_L}]\cdot [D_1]&=2k-r\,,\\
			[\omega_{\mc Q_L}]\cdot [D_3]&=-r(d+(r-k)(1-g))\,.
		\end{align*}
		The proof of Proposition \ref{alpha nef} shows that $[\alpha]\cdot [D_1]=0$.
		Using Lemma \ref{D_2 dot alpha} and Lemma \ref{alpha dot D_3}, we get
		the following two equations in $a$ and $b$
		\begin{align*}
			b&=2k-r\,,\\
			a(r-k)+bd(r-k-1)&=-r(d+(r-k)(1-g))\,.
		\end{align*}
		One easily computes that $a=d(r-2k-2)+r(g-1)$.
		Thus, 
		$$\omega_{\mc Q_L}=[d(r-2k-2)+r(g-1)]\alpha+(2k-r)\beta_{d+g-1}\,.$$
		
		For $\mc Q_L$ to be Fano, we need that $d(r-2k-2)+r(g-1)<0$ and $2k-r<0$.
		Since $d\gg0$, this happens iff $r-2k-2<0$ and $2k-r<0$, that is, 
		iff $2k<r<2k+2$, that is, iff $r=2k+1$. 
	\end{proof}

	This completes the proof of Theorem \ref{thm intro}.


\end{document}